\documentclass[letter, 10 pt, conference]{ieeeconf}  

\IEEEoverridecommandlockouts                              

\usepackage{amsmath} 

\usepackage{gensymb}
\usepackage{amsthm}
\usepackage{amssymb}  
\usepackage{stfloats}
\usepackage[noadjust]{cite}
\usepackage{graphicx}
\usepackage{enumerate}
\graphicspath{{./pics/}}

\theoremstyle{definition}
\newtheorem{definition}{Definition}[section]
\newtheorem{theorem}{Theorem}[section]
\newtheorem{corollary}{Corollary}[theorem]
\newtheorem{lemma}[theorem]{Lemma}
\theoremstyle{remark}
\newtheorem*{remark}{Remark}

\title{\LARGE \bf
	Output Feedback Controller Synthesis for Negative Imaginary Systems
}

\author{James Dannatt$^{1}$, Ian Petersen$^{2}$
\thanks{*This work was supported by the Australian Research Council under grants DP160101121 and DP190102158.}
\thanks{$^{1}$James Dannatt is with the Research School of Engineering at the Australian National University, ACT.
        {\tt\small james.dannatt@anu.edu.au}.}%
\thanks{$^{2}$Ian Petersen is with the Research School of Engineering at the Australian National University, ACT.
        {\tt\small i.r.petersen@gmail.com}.}%
}

\begin{document}

\maketitle
\thispagestyle{empty}
\pagestyle{empty}

\begin{abstract}
This paper presents necessary and sufficient conditions for deriving a strictly proper dynamic controller which satisfies the negative imaginary output feedback control problem. Our synthesis method divides the output feedback control problem into a state feedback problem and a dual output injection problem. Thus, a controller is formed using the solutions to a pair of dual algebraic Riccati equations. Finally, an illustrative example is offered to show how this method may be applied.
\end{abstract}

\section{Introduction}

Negative imaginary (NI) systems theory is concerned with stable systems with a relative degree of zero, one and two that have a phase response in the interval $[-\pi, 0]$ for non negative frequencies~\cite{Petersen2010}. These systems were first introduced in \cite{Lanzon2008}, motivated by the study of linear mechanical systems with collocated force inputs and position outputs. Since the introduction of NI systems theory, it has found use in a large number of applications. A summary of many of these applications can be found in \cite{petersen2016negative}. Along with introducing a definition of NI systems, \cite{Lanzon2008} showed that the positive feedback interconnection of an NI system with a strictly negative imaginary (SNI) system is internally stable as long as the closed-loop dc gain is less than unity. Thus, if the plant uncertainty of a system is known to be SNI, a feedback controller can be constructed such that the closed-loop system is NI. The resulting positive feedback interconnection can then guaranteed to be robustly stable under the appropriate DC gain condition~\cite{Petersen2010}. The construction of such a controller has prompted the development of both state and output feedback results.
\\\\
The state feedback control problem for negative imaginary systems is concerned with designing a controller $K$ for the positive feedback control scheme $u=Kx$. $K$ must be chosen such that the corresponding closed-loop system has the negative imaginary property when all of the systems state-variables are available for feedback. The earliest NI state feedback results were presented in \cite{Petersen2010} and \cite{5430931} using the solution of an LMI as the basis for controller design. \cite{Mabrok2012} then drew on the $H_{\infty}$ literature\footnote{See \cite{petersen1991first} and \cite{petersen1989complete}.}, in order to propose a state feedback controller synthesis method that used the solution to an ARE to form a controller. Unfortunately, neither of these methods were ideal as they allowed the closed-loop system to have poles at the origin. In fact an origin pole was unavoidable when a controller was designed using the method in \cite{Mabrok2012}. Recognizing this, \cite{Mabrok2012a,Mabrok2015} applied a perturbation of the plant matrix in the design process to guarantee asymptotic stability of the closed-loop system. This method ensured stability, however only sufficient conditions were offered and no proof was provided ensuring the perturbed system maintained the NI property. This was rectified in \cite{arXiv:1807.07212} were both necessary and sufficient conditions for state feedback controller synthesis were offered. The closed-loop system in \cite{arXiv:1807.07212} was shown to have both the NI property and a prescribed degree of stability.
\\\\
In practice the full system state is not always available for use. The negative imaginary output feedback control problem is then finding a controller $H$ for the control scheme $u=H\hat{x}$, where $\hat{x}$ is an estimate of the systems state-variables. $H$ is chosen such that the corresponding closed-loop system has the negative imaginary property. This problem was first addressed in \cite{xiong2016output} where both static and dynamic output controllers were offered. Alternatively, \cite{gab_negimagsynth,salcan2019negative} presented a method of controller synthesis using the solution to dual AREs as in \cite{sun1994solution}. In each of these cases, only sufficient conditions were offered.
\\\\
Here, we present necessary and sufficient conditions for deriving a dynamic controller that solves the output feedback control problem. Our approach uses the solution to two dual AREs like in \cite{gab_negimagsynth,petersen1991first,sun1994solution}. In our method, the solutions of our AREs may be obtained through the Schur decomposition of a matrix and the solution of two Lyapunov equations, avoiding the common problem of singular Hamiltonians associated with the NI AREs. Our controller results in a closed-loop system with the NI property. Thus, if the plant uncertainty is SNI, the resulting positive feedback interconnection will be robustly stable when the DC gain condition of \cite{Lanzon2008} is satisfied.
\\\\
To assist with exposition, proofs may be found in the appendix.

\section{Notation}

Let $\mathbb{R}$ and $\mathbb{C}$ denote the fields of real and complex numbers respectively. The notation $\operatorname{Im}[G(j\omega)]$ refers to the imaginary component of the frequency response $G(j\omega)$. Analogously $\operatorname{Re}[G(j\omega)]$ refers to the real component of $G(j\omega)$. $C^*$ refers to the complex conjugate transpose of a matrix or vector $C$ and ${R}^{\sim}(s)$ represents $R^T(-s)$. The notation $\rho(A)$ denotes the spectral radius of $A$.

\section{Preliminaries}

The following section provides a review of the relevant NI and SNI definitions and lemmas needed in proving our main result.
\\\\
Consider the linear time-invariant (LTI) system described by
\begin{IEEEeqnarray}{c} \label{system}
	\dot{x}(t) = Ax(t) + Bu(t) \nonumber \\
	y(t) = Cx(t) + Du(t)
\end{IEEEeqnarray}
where $A \in \mathbb{R}^{n \times n},$ $B \in \mathbb{R}^{n \times m},$ $C \in \mathbb{R}^{m \times n}$ and $D \in \mathbb{R}^{m \times m}$.
We denote the proper, real, rational transfer function matrix of this system as $G(s) = C(sI-A)^{-1}B + D$.

\begin{definition} \label{def: dual definition} [Dual System]
For the system (\ref{system}), the transpose of this system is defined as
\begin{IEEEeqnarray}{c} \label{system dual}
	\dot{x}(t) = A^Tx(t) + C^Tu(t) \nonumber \\
	y(t) = B^Tx(t) + D^Tu(t)
\end{IEEEeqnarray}
and is denoted by the proper, real-rational transfer function matrix $G(s)^T = B^T(sI-A^T)^{-1}C^T + D^T$.
\end{definition}

We now provide a formal definition of both negative imaginary and strictly negative imaginary systems.

\theoremstyle{definition}
\begin{definition} \label{def: NI definition}{ [Negative Imaginary] \cite{mabrok2011stability}\\ A square transfer function matrix $G(s)$ is NI if the following conditions are satisfied:}
	\begin{enumerate}
		\item $G(s)$ has no pole in $\operatorname{Re}[s]>0$.

		\item For all $\omega \geq 0$ such that $jw$ is not a pole of $G(s)$, $j(G(j\omega) - G(j\omega)^*) \geq 0$.

		\item If $s=j\omega_0$, $ \omega_0 > 0$ is a pole of $G(s)$ then it is a simple pole. Furthermore, if $s=j\omega_0$, $ \omega_0 > 0$ is a pole of $G(s)$, then the residual matrix $K = \lim_{s \to j\omega_0} (s-j\omega_0)jG(s)$ is positive semidefinite Hermitian.

		\item If $s=0$ is a pole of $G(s)$, then it is either a simple pole or a double pole. If it is a double pole, then, $\lim_{s \to 0} s^2G(s) \geq 0$.
	\end{enumerate}
\end{definition}

Also, an LTI system (\ref{system}) is said to be NI if the corresponding transfer function matrix $G(s) = C(sI-A)^{-1}B + D$ is NI.

\begin{definition} \label{def: SNI definition}
{[Strictly Negative Imaginary] \cite{mabrok2011stability,8287451} \\ A square transfer function matrix $G(s)$ is SNI if the following conditions are satisfied:}
	\begin{enumerate}
		\item $G(s)$ has no poles in $\operatorname{Re}[s] \geq 0$.

		\item For all $\omega > 0$ such that $j\omega$ is not a pole of $G(s)$, $j(G(j\omega) - G(j\omega)^*) > 0$.
	\end{enumerate}
\end{definition}

Also, an LTI system (\ref{system}) is said to be SNI if the corresponding transfer function matrix $G(s) = C(sI-A)^{-1}B + D$ is SNI.

\begin{lemma} \label{lemma: G NI implies sH PR} [$G(s)$ is NI $\iff$ $s(G(s)-D)$ is PR] \cite{patra2011stability} \\ A square, real, rational and proper transfer function matrix $G(s)$ is NI if and only if $F(s) = s(G(s)-D)$ is positive real (PR).
\end{lemma}

The following lemmas are required in the proof of our main result.

\begin{lemma} \label{lemma: Gs NI iff Gs^T NI definition} [$G(s)$ is NI $\iff$ $G(s)^T$ is NI] \cite{gab_negimagsynth} \\ A square, real, rational and proper transfer function matrix $G(s)$ is NI (respectively SNI) if and only if $G(s)^T$ is NI (respectively SNI).
\end{lemma}

\begin{lemma}\label{lemma: NI ARE lemma} [ARE Negative Imaginary Lemma] \cite{Mabrok2015} \\
Let $
\begin{bmatrix}
    \begin{tabular}{ l | r }
  $A$ & $B$ \\ \hline
  $C$ & $D$
\end{tabular}
\end{bmatrix}
$ be a minimal realization of a real, rational transfer function matrix $G(s)$ and suppose $CB + B^TC^T > 0$. Then the following statements are equivalent:
\begin{enumerate}
\item $G(s)$ is NI.
\item The ARE
\begin{align} \label{math: NI ARE}
	PA + A^TP + (CA-B^TP)^TR^{-1}(CA-B^TP) = 0
\end{align} has a positive semi-definite solution $P \geq 0$.
\item The ARE
\begin{align} \label{math: NI ARE DUAL}
	ZA^T + AZ + (B-ZA^TC^T)R^{-1}(B^T-CAZ) = 0
\end{align} has a positive semi-definite solution $Z \geq 0$.
\end{enumerate}
\end{lemma}

\begin{lemma} \label{lemma: SNI ARE lemma} [ARE Strictly Negative Imaginary Lemma] \cite{arXiv:1807.07212}
Let $
\begin{bmatrix}
    \begin{tabular}{ l | r }
  $A$ & $B$ \\ \hline
  $C$ & $D$
\end{tabular}
\end{bmatrix}
$ be a minimal realization of a real, rational transfer function matrix $G(s)$ and suppose $CB + B^TC^T > 0$. Then the following statements are equivalent:
\begin{enumerate}
\item $G(s)$ is SNI and $A$ has no imaginary-axis eigenvalues.
\item The ARE
\begin{align} \label{math: SNI ARE}
	PA + A^TP + (CA-B^TP)^TR^{-1}(CA-B^TP) = 0
\end{align} has a positive definite solution $P>0$ and all the eigenvalues of the matrix $A-BR^{-1}C(A-B^TP)$ lie in the open left half of the complex plane or at the origin.
\item The ARE
\begin{align} \label{math: SNI ARE DUAL}
	ZA^T + AZ + (B-ZA^TC^T)R^{-1}(B^T-CAZ) = 0
\end{align} has a positive definite solution $Z>0$ all the eigenvalues of the matrix $A-(B-YA^TC^T)R^{-1}CA$ lie in the open left half of the complex plane or at the origin.
\end{enumerate}
\end{lemma}

\begin{lemma}~\label{lemma: Z equals P inverse}
If a matrix matrix $P>0$ solves the ARE (\ref{math: SNI ARE}), then the matrix $Z = P^{-1}$ is a solution to the ARE
\begin{align*}
 Z{A_0}^T + A_0Z + Z\bar{Q}Z + BR^{-1}B^T &= 0,
\end{align*}
where
\begin{IEEEeqnarray*}{c}
    A_0 = A-BR^{-1}CA, \\
    R = CB + B^TC^T, \\
    \bar{Q} = A^TC^TR^{-1}CA.
  \end{IEEEeqnarray*}
\end{lemma}

\begin{definition} \cite{zhou1996robust}
Consider the real, rational, LTI system (\ref{system}) with transfer function $G(s)$. The operation of the system by
\begin{align*}
\dot{x} = Ax + Bu \longmapsto \dot{x} = Ax + Bu + Ly
\end{align*}
is called output injection and can be written as
\begin{align*}
\begin{bmatrix}
    \begin{tabular}{ l | r }
  $A$ & $B$ \\ \hline
  $C$ & $D$
\end{tabular}
\end{bmatrix} \mapsto
\begin{bmatrix}
    \begin{tabular}{ c | c }
  $A + LC$ & $B + LD$ \\ \hline
  $C$ & $D$
\end{tabular}
\end{bmatrix}.
\end{align*}
\end{definition}
Output injection does not change the detectability of a system \cite{zhou1996robust}.

Before the output feedback control problem can be addressed, the state feedback control problem for negative imaginary systems must be discussed.

\section{static state feedback}

Consider the state space representation of a linear uncertain system described by
\begin{align} \label{uncertain system}
  \dot{x} &= Ax + B_1w + B_2u \\
  z       &= C_1x \label{uncertain systemb} \\
  y       &= C_2x + D_{21}w   \label{uncertain systemc}
\end{align}
This system has uncertainty $\Delta(s)$ with  state space representation:
\begin{align} \label{uncertainty model}
  \dot{x}_{\Delta} &= A_{\Delta}x_{\Delta} + B_{\Delta}z \\
  w       &= C_{\Delta}x_{\Delta} + D_{\Delta}z \label{uncertainty modelc}
\end{align}
Also, assume that all of this system's state-variables are available for feedback.
\\\\
The state feedback control problem for negative imaginary systems is then, under the control scheme $u=Kx$, can we design a controller $K$ such that the corresponding closed-loop uncertain system
\begin{align} \label{math: closed loop system}
  \dot{x} &= (A + B_2K)x + B_1w \\
  z       &= C_1x        \label{math: closed loop systemb}     \\
  \dot{x}_{\Delta} &= A_{\Delta}x_{\Delta} + B_{\Delta}z \\
  w       &= C_{\Delta}x_{\Delta} + D_{\Delta}z
\end{align}
has the NI property.
\\\\
The following lemma is one solution to this problem.

\begin{lemma}\label{lemma: Dannatt/petersen synthesis lemma} [NI State Feedback Control Lemma] \cite{Mabrok2012,Mabrok2015,arXiv:1807.07212} Consider the uncertain system (\ref{uncertain system})-(\ref{uncertain systemb}) that satisfies $C_1B_2$ non-singular and $C_1B_1 + B_1^TC_1^T > 0$. The following statements are equivalent:
\begin{enumerate}
\item There exists a static state feedback matrix $K$ such that the closed-loop system (\ref{math: closed loop system})-(\ref{math: closed loop systemb}) is NI.
\item There exists matrices $T \geq 0$ and $S \geq 0$ such that
\begin{align}
  -A_{22}T - TA_{22}^T + B_{f2}RB_{f2}^T &= 0, \label{math: T cond}\\
  -A_{22}S - SA_{22}^T + B_{22}R^{-1}B_{22}^T &= 0, \label{math: S cond}\\
  T-S &> 0,
\end{align}
where the matrices $A_{22}$, $B_{f2}$ and $B_{22}$ are obtained from the Schur decomposition outlined below in (\ref{math: schur decomp s})-(\ref{math: schur decomp e}).
\item There exists a positive semi-definite matrix $P_f\geq0$ that satisfies the ARE
\begin{align*}
P_fA_f + A_f^TP_f +P_f(\tilde{B_1}R^{-1}\tilde{B_1}^T-B_fRB_f^T)P_f = 0,
\end{align*}
where $A_f$,$B_f$ and $\tilde{B_1}$ are defined below in (\ref{math: schur decomp s})-(\ref{math: schur decomp e}).
\end{enumerate}
If any of the above statements hold, then a corresponding state feedback controller matrix $K$ is given by,
\begin{align*}
K &= (C_1B_2)^{-1}(B_1^TP - C_1A - R(B_2^TC_1^T)^{-1}B_2^TP),
\end{align*}
where $P = UP_fU^T$ and $P_f = \begin{bmatrix}
    \begin{tabular}{ l  r }
  $0$ & $0$ \\
  $0$ & $(T-S)^{-1}$
\end{tabular}
\end{bmatrix} \geq 0$.
$U$ is an orthogonal matrix obtained through the following real Schur transformation~\footnote{See Section 5.4 of \cite{Bernstein2009}}.
\end{lemma}

\subsection{Schur decomposition}

Consider the following real Schur transformation of the matrix $A-B_2(C_1B_2)^{-1}C_1A$ which we can apply to the system (\ref{math: closed loop system}), (\ref{math: closed loop systemb}) to give

\begin{align} \label{math: schur decomp s}
  A_f &= U^T(A-B_2(C_1B_2)^{-1}C_1A)U   =
  \begin{bmatrix}
    \begin{tabular}{ l  r }
  $A_{11}$ & $A_{12}$ \\
  0 & $A_{22}$
\end{tabular}
\end{bmatrix},
\\
  B_f &= U^T(B_2(C_1B_2)^{-1} - B_1R^{-1}) =
  \begin{bmatrix}
    \begin{tabular}{c}
      $B_{f1}$ \\
      $B_{f1}$
\end{tabular}
\end{bmatrix},\\
  \tilde{B_1} &= U^TB_1 =
  \begin{bmatrix}
    \begin{tabular}{c}
      $B_{11}$ \\
      $B_{22}$
\end{tabular}
\end{bmatrix}. \label{math: schur decomp e}
\end{align}
This transformation is constructed such that all of the eigenvalues of the matrix $A_{11}$ are in the closed left half plane and $A_{22}$ is an anti-stable matrix.

\begin{remark}
The sub-matrix $A_{11}$ will always contain a zero eigenvalue. This follows directly from that fact that $A-B_2(C_1B_2)^{-1}C_1A$ is singular.
\end{remark}

\begin{remark}
If the state feedback controller $K$ is applied to a system for which the plant uncertainty is known to be SNI and the DC gain condition of \cite{Lanzon2008} is satisfied, then the resulting positive feedback interconnection of the plant uncertainty with the closed-loop transfer function is guaranteed to be robustly stable.
\end{remark}

\begin{lemma}\label{lemma: dual output injection problem} [NI Dual Output Injection Lemma]
Consider the uncertain system (\ref{uncertain system})-(\ref{uncertain systemc}) that satisfies $D_{21}$ non-singular and $C_1B_1 + B_1^TC_1^T > 0$. The following statements are equivalent:
\begin{enumerate}
\item The closed-loop system (\ref{math: closed loop system})-(\ref{math: closed loop systemb}) is NI.
\item There exists an output injection matrix $L$ such that the system
\begin{align} \label{math: dual closed loop system}
  \dot{x}   &= (A + LC_2)^Tx + A^TC_1^Tw \\
  \tilde{z} &= (B_1^T + D_{21}^TL^T)x + B_1^TC_1^T \label{math: dual closed loop systemb}
\end{align}
is PR.
\item There exists a positive semi-definite matrix $Z\geq0$ that satisfies the ARE
\begin{align*}
Z(A + LC_2)^T + (A + LC_2)Z + Q^TR^{-1}Q = 0,
\end{align*}
where $Q = B_1^T+D_{21}^TL^T-C_1AZ$.
\end{enumerate}
If any of the above statements hold, a corresponding control matrix $L$ that solves this systems dual output injection problem is given by
\begin{align} \label{math: definition for L}
L = (ZA^TC_1^T - B_1 - ZC_2^T(D_{21})^{-1}R)(D_{21}^T)^{-1}.
\end{align}
\end{lemma}

We will now use both the state feedback and output injection lemmas in constructing a controller that satisfies the output feedback control problem.

\section{Negative imaginary dynamic output feedback} \label{sec: main_result}

Consider the state space representation of a linear uncertain system
\begin{align} \label{system: main a}
  \dot{x} &= Ax + B_1w + B_2u \\
  z       &= C_1x            \\
  y       &= C_2x + D_{21}w             \label{system: main c}
\end{align}
where $A \in \mathbb{R}^{n \times n}$, $B_1 \in \mathbb{R}^{n \times 1}$, $B_2 \in \mathbb{R}^{n \times r}$, $C_1 \in \mathbb{R}^{1 \times n}$, $C_2 \in \mathbb{R}^{1 \times n}$, $D_{21} \in \mathbb{R}^{n \times n}$. This system is assumed to have SNI uncertainty $\Delta(s)$ with state space representation (\ref{uncertainty model})-(\ref{uncertainty modelc}).
\\\\
Suppose the system (\ref{system: main a})-(\ref{system: main c}) satisfies the following assumptions:
\begin{enumerate}[{A}1.]
 \item $(A,B_2)$ is stabilizable and $(C_2,A)$ is detectable;
 \item $C_1B_2$ is non-singular;
 \item $D_{21}$ is non-singular;
 \item $R = C_1B_1 + B_1^TC_1^T > 0$. \label{Last Asummption}
\end{enumerate}

If we apply the dynamic compensator
\begin{align}
  \dot{x_k} &= A_kx_k + B_ky \label{controller: xdot}\\
  u      &= C_kx_k \label{controller: y}
\end{align}
to the system (\ref{system: main a})-(\ref{system: main c}) the corresponding closed-loop system has the realization
\begin{align} \label{closed loop realization x.}
\begin{bmatrix}
    \begin{tabular}{ c c | r }
  $A$ & $B_2C_k$ & $B_1$ \\
  $B_kC_2$ & $A_k$ & $B_kD_{21}$ \\ \hline
  $C_1$ & $0$ & $0$
\end{tabular}
\end{bmatrix},
\end{align}
with closed-loop transfer function
\begin{align} \label{math: closed-loop transfer function}
  G_{cl}(s) &= C_c\big(sI - A_c\big)^{-1}B_c.
\end{align}

We now present both necessary and sufficient conditions for the existence of a dynamic output controller that will result in a closed-loop system with the NI property.
%
\begin{theorem}~\label{theorem: NI outputfeedback sufficiency}
Consider the uncertain system (\ref{system: main a})-(\ref{system: main c}) satisfying assumptions A1-A\ref{Last Asummption}. Suppose there exist $P\geq0$, $Z\geq0$, $F$ and $L$ which satisfy:
\begin{enumerate}[(a)]
 \item
 \begin{IEEEeqnarray}{c}
  R(P) = P\tilde{A} + \tilde{A}^TP + \tilde{Q}^TR^{-1}\tilde{Q} = 0,
  \end{IEEEeqnarray}
where
\begin{align*}
    \tilde{A} &= A + B_2F , &
    \tilde{Q} &= C_1(A+B_2F)-B_1^TP.
  \end{align*}
  \item
 \begin{IEEEeqnarray}{c}
  S(Z) = Z\bar{A}^T + \bar{A}Z + \bar{Q}R^{-1}\bar{Q}^T = 0,
  \end{IEEEeqnarray}
  where
\begin{align*}
    \bar{A} &= A + LC_2, &
    \bar{Q} &= B_1+LD_{21}-ZA^TC_1^T.
  \end{align*}
\item $\rho (ZP) < 1$.
\end{enumerate}
Then, a strictly proper controller which results in a closed-loop system with the negative imaginary property is given by (\ref{controller: xdot}), (\ref{controller: y}) where
\begin{align}
    C_k &= F, \label{controller: C_k choice}\\
    B_k &= -(I-ZP)^{-1}L, \label{controller: B_k choice} \\
    A_k &= A+B_2C_k-B_kC_2-\big(B_1 + (I-ZP)^{-1}LD_{21}\big) \nonumber\\
    & \omit\hfill ${}  \cdot R^{-1}\big(C_1(A+B_2C_k)-B_1^TP\big).$ \nonumber\\
    \label{controller: A_k choice}
  \end{align}

Conversely, suppose there exists a controller of the form (\ref{controller: xdot}), (\ref{controller: y}) such that the closed-loop system is SNI and minimal. Then there exists $P>0$, $Z>0$, $F$ and $L$ which satisfy
\begin{enumerate}[(a)]
 \item
 \begin{IEEEeqnarray}{c}
  R(P) = P\tilde{A} + \tilde{A}^TP + \tilde{Q}^TR^{-1}\tilde{Q} = 0,
  \end{IEEEeqnarray}
where
\begin{align*}
    \tilde{A} &= A + B_2F , &
    \tilde{Q} &= C_1(A+B_2F)-B_1^TP.
  \end{align*}
  \item
 \begin{IEEEeqnarray}{c}
  S(Z) = Z\bar{A}^T + \bar{A}Z + \bar{Q}R^{-1}\bar{Q}^T = 0,
  \end{IEEEeqnarray}
  where
\begin{align*}
    \bar{A} &= A + LC_2, &
    \bar{Q} &= B_1+LD_{21}-ZA^TC_1^T.
  \end{align*}
\item $\rho (ZP) \leq 1$.
\end{enumerate}
\end{theorem}

\begin{corollary} \cite{salcan2019negative}
Consider the uncertain system (\ref{system: main a})-(\ref{system: main c}) satisfying assumptions A1-A\ref{Last Asummption} and suppose there exist $P>0$, $Z>0$, $F$ and $L$ which satisfy Theorem~\ref{theorem: NI outputfeedback sufficiency}. If the matrix $A_c$ is Hurwitz, then the resulting closed-loop system is SNI.
\end{corollary}

\begin{remark}
The controller (\ref{controller: xdot})-(\ref{controller: y}) is applied to a system for which the plant uncertainty is assumed to be SNI. If the dc gain condition of \cite{Lanzon2008} is satisfied, then the resulting positive feedback interconnection of the plant uncertainty with the closed-loop transfer function is guaranteed to be robustly stable.
\end{remark}

The following lemma is needed in the proof of Theorem~\ref{theorem: NI outputfeedback sufficiency}.

\begin{corollary}\label{collary: Building V}
Consider the uncertain system (\ref{system: main a})-(\ref{system: main c}) satisfying assumptions A1-A\ref{Last Asummption}. If there exists $P\geq0$, $Z\geq0$, $F$ and $L$ which satisfy Theorem~\ref{theorem: NI outputfeedback sufficiency}, then there exists a matrix $V\geq0$ which satisfies the ARE
\begin{align}\label{math: ARE using V}
  V(A_e + L_eC_{e2}) + (A_e + L_eC_{e2})^TV + Q_v^TR^{-1}Q_v = 0,
\end{align}
where
\begin{align*}
  A_e &= A - B_1R^{-1}C_1(A+B_2F) + B_1R^{-1}B_1^TP, \\
  C_{e1} &= C_1B_2F, \\
  C_{e2} &= C_2 - D_{21}R^{-1}C_1(A+B_2F) + D_{21}R^{-1}B_1^TP, \\
  Q_v &= C_{e1} -(B_1+L_eD_{21})V.
\end{align*}
\end{corollary}

\section{Illustrative example}

Consider the linear uncertain system (\ref{system: main a})-(\ref{system: main c}), where
\begin{align} \label{example system}
A =  \begin{bmatrix}
    \begin{tabular}{ c c c }
$-1$ & $0$ & $0$ \\
$0$ & $-1$ & $1$ \\
 $1$ & $1$ & $-1$
    \end{tabular}
\end{bmatrix}, \
B_1 = \begin{bmatrix}
    \begin{tabular}{ c }
    $1$ \\
    $2$  \\
    $1$
    \end{tabular}
\end{bmatrix}, \  B_2 = \begin{bmatrix}
    \begin{tabular}{ c }
    $1$ \\
    $1$  \\
    $1$
    \end{tabular}
\end{bmatrix}, \nonumber\\
C_1 = \begin{bmatrix}
    \begin{tabular}{c c c}
    $1$ & $0$  & $0$
    \end{tabular}
\end{bmatrix}, \
C_2 = \begin{bmatrix}
    \begin{tabular}{c c c}
    $1$ & $2$  & $0$
    \end{tabular}
\end{bmatrix}, \
D_{21} = I.
\end{align}
We can easily verify this system satisfies A1-A\ref{Last Asummption}. Also, the matrices $P=0$ and $Z=0$ satisfy conditions (a)-(c) of Theorem~\ref{theorem: NI outputfeedback sufficiency}. Thus we can construct matrices $F= \begin{bmatrix}
    \begin{tabular}{c c c}
    $1$ & $0$  & $0$
    \end{tabular}
\end{bmatrix}$ and $L = \begin{bmatrix}
    \begin{tabular}{c c c}
    $-1$ & $-2$  & $-1$
    \end{tabular}
\end{bmatrix}^T$ and our dynamic compensator with the form (\ref{controller: xdot}), (\ref{controller: y}) can be constructed as
\begin{align*}
  \begin{bmatrix}
    \begin{tabular}{ c c c | c}
$-1$ & $-2$ & $0$ & $1$ \\
$-1$ & $-5$ & $1$ & $2$ \\
 $1$ & $-1$ & $-1$ & $1$ \\ \hline
 $1$ & $0$ & $0$ & $0$
    \end{tabular}
\end{bmatrix}.
\end{align*}

We can now verify our closed-loop system with transfer function $G_{cl}(s) = C_{cl}\big(sI - A_{cl}\big)^{-1}B_{cl}$, where
\begin{align*}
A_{cl} &=  \begin{bmatrix}
    \begin{tabular}{ c c c c c c}
$0$ & $0$ & $0$ & $-1$ & $0$ & $0$ \\
$1$ & $-1$ & $1$ & $-1$ & $0$ & $0$ \\
 $2$ & $1$ & $-1$ & $-1$ & $0$ & $0$ \\
 $0$ & $0$ & $0$ & $-2$ & $-2$ & $0$ \\
$0$ & $0$ & $0$ & $-2$ & $-5$ & $1$ \\
 $0$ & $0$ & $0$ & $0$ & $-1$ & $-1$
    \end{tabular}
\end{bmatrix},
B_{cl} = \begin{bmatrix}
    \begin{tabular}{ c }
    $1$ \\
    $2$  \\
    $1$ \\
    $0$ \\
    $0$  \\
    $0$
    \end{tabular}
\end{bmatrix}, \\
C_{cl} &= \begin{bmatrix}
    \begin{tabular}{c c c c c c}
    $1$ & $0$  & $0$ & $0$  & $0$ & $0$
    \end{tabular}
\end{bmatrix},
\end{align*}
satisfies the the ARE
\begin{align*}
\Sigma A_{cl} + A_{cl}^T\Sigma \nonumber\\
& \omit\hfill ${} + (C_{cl}A_{cl}-B_{cl}^T\Sigma)^TR^{-1}(C_{cl}A_{cl} - B_{cl}^T\Sigma) = 0,$
\end{align*}
with
\begin{align*}
 \Sigma &=  \begin{bmatrix}
        \begin{tabular}{ c c c c c c}
  $0$ & $0$ & $0$ & $0$ & $0$ & $0$ \\
  $0$ & $0$ & $0$ & $0$ & $0$ & $0$ \\
  $0$ & $0$ & $0$ & $0$ & $0$ & $0$ \\
  $0$ & $0$ & $0$ & $0.178$ & $-0.053$ & $-0.024$ \\
  $0$ & $0$ & $0$ & $-0.053$ & $0.019$ & $0.010$ \\
  $0$ & $0$ & $0$ & $-0.024$ & $0.010$ & $0.010$
    \end{tabular}
    \end{bmatrix} \geq 0.
\end{align*}
Therefore it follows from Lemma~\ref{lemma: NI ARE lemma} that $G_{cl}(s)$ is NI.

\section{Conclusions}

This paper has presented both necessary and sufficient conditions for synthesizing a dynamic controller which solves the negative imaginary output feedback control problem. Our method divides the output feedback control problem into a state feedback problem and a dual output injection problem. These problems are then solved by finding the solutions to a pair of algebraic Riccati equations. The solutions to these equations may also be obtained through Schur decomposition which avoids the common problem of singular Hamiltonians in NI controller synthesis. Relaxing the necessary conditions for NI controller synthesis to include non strict solutions to the NI ARE remains an open problem.

\section{Appendix: Proofs}

\begin{proof}[Proof of Lemma~\ref{lemma: NI ARE lemma}]
The proof of equivalence between 1) and 2) is shown in \cite{Mabrok2015}.
1) $\iff$ 3) follows using the same proof applied to the transfer function matrix $G(s)^T$ which is NI if and only if $G(s)$ is NI by Lemma~\ref{lemma: Gs NI iff Gs^T NI definition}.
\end{proof}

\begin{proof}[Proof of Lemma~\ref{lemma: SNI ARE lemma}]
The proof of equivalence between 1) and 2) is shown in \cite{Mabrok2015}.
1) $\iff$ 3) follows using the same proof applied to the transfer function matrix $G(s)^T$ which is SNI if and only if $G(s)$ is SNI by Lemma~\ref{lemma: Gs NI iff Gs^T NI definition}. Also, note that $G(s)$ is SNI implies $A^T$ has no imaginary-axis eigenvalues from Definition~\ref{def: SNI definition}.
\end{proof}

\begin{proof}[Proof of Lemma~\ref{lemma: Z equals P inverse}]
This is proved by straight forward algebraic manipulation. Noting that $Z=P^{-1}$ we simply pre-multiply (\ref{math: NI ARE}) by $Z$ and post-multiply by $Z$ as follows
\begin{align*}
  0 &= PA + A^TP + (CA-B^TP)^TR^{-1}(CA-B^TP), \\
  0 &= Z\bigg(PA + A^TP + (CA-B^TP)^TR^{-1} \\
  & \omit\hfill ${} \cdot(CA-B^TP)\bigg)Z$, \\
  0 &= Z{A}^T + AZ +(B^T-CAZ)^TR^{-1}(B^T-CAZ).
  \end{align*}
\end{proof}

\begin{proof}[Proof of Lemma~\ref{lemma: dual output injection problem}]
The equivalence between $2) \iff 3)$ follows from Lemma 2.3 in \cite{sun1994solution}.
In order to show $1) \iff 2)$ note that from Lemma~\ref{lemma: G NI implies sH PR}, condition 1) holds if and only if $s(G(s)-D)$ is PR. If we then consider the Dual Positive Real Lemma in \cite{anderson1967dual}, $1) \iff 2)$ follows directly.
\end{proof}

\begin{proof}[Proof of Corollary~\ref{collary: Building V}]
This proof is analogous to finding the W matrix in the sufficiency proof of Theorem 3.1 in \cite{petersen1991first}.
\\
It follows from condition (b) of Theorem~\ref{theorem: NI outputfeedback sufficiency} that the ARE (\ref{math: NI ARE DUAL}) has solution $Z\geq0$. Let $A_w = A - B_1D_{21}^{-1}C_2$ and $R_w = PA_w + A_w^TP + R_z$,
where
\begin{multline*}
 R_z = A^TC_1^TD_{21}^{-1}C_2 + (A^TC_1^TD_{21}^{-1}C_2)^T \\
 - (D_{21}^{-1}C_2)^TR(D_{21}^{-1}C_2).
\end{multline*}
We can manipulate (\ref{math: NI ARE DUAL}) as follows
\begin{align}
  Z\bar{A}^T + \bar{A}Z + \bar{Q}R^{-1}\bar{Q}^T &= 0,  \nonumber\\
  (A - B_1D_{21}^{-1}C_2)Z + Z(A - B_1D_{21}^{-1}C_2) + R_z &= 0, \nonumber\\
  A_wZ + ZA_w^T + Z(R_w - PA_w - A_w^TP)Z &= 0,  \nonumber\\
  A_wZ - ZPA_wZ + ZA_w^T - ZA_w^TPZ + ZR_wZ &= 0,  \nonumber\\
  (I - ZP)A_wZ + ZA_w^T(I-PZ) + ZR_wZ &= 0, \nonumber\\
  A_wW + WA_w^T + WR_w &= 0, \label{math: w equation}
\end{align}
  where
\begin{align*}
    \bar{A} &= A + LC_2, &
    \bar{Q} &= B_1+LD_{21}-ZA^TC_1^T.
\end{align*}
It follows from condition (c) of Theorem~\ref{theorem: NI outputfeedback sufficiency} that $W=Z(I-PZ)^{-1}$ is positive semi-definite and well defined. After algebraic manipulation we can rewrite $A_w$ and $R_w$ as
\begin{align*}
    A_w &= A_e + B_1R^{-1}C_{e1} - B_1D_{21}^{-1}(C_{e2}+D_{21}R^{-1}C_{e1}), \\
    R_w &= C_{e1}^TR^{-1}C_{e1} \\
    & \omit\hfill ${}  - (D_{21}^{-1}C_{e2}+RC_{e1})^TR^{-1}(D_{21}^{-1}C_{e2}+RC_{e1}),$
\end{align*}
where
\begin{align*}
    A_e &= A - B_1R^{-1}C_1(A+B_2F) + B_1R^{-1}B_1^TP, \\
    C_{e1} &= C_1B_2F, \\
   C_{e2} &= C_2 - D_{21}R^{-1}C_1(A+B_2F) + D_{21}R^{-1}B_1^TP. \\
\end{align*}
If we then substitute these into (\ref{math: w equation}), we see that $W\geq0$ is a solution to the ARE
\begin{align*}
  W(A_e + L_eC_{e2})^T + (A_e + L_eC_{e2})W + Q_eR^{-1}Q_e^T = 0,
\end{align*}
where
\begin{align*}
 L_e &= -\big( B_1R^{-1} + W(D_{21}^{-1}C_{e2}-R^{-1}C_{e1})^T \big)\big(R^{-1}D_{21}^T\big)^{-1}, \\
 Q_e &= B_1+L_eD_{21}-WC_{e1}^T.
\end{align*}
It then follows from Lemma~\ref{lemma: dual output injection problem} that the system
\begin{align*}
\begin{bmatrix}
    \begin{tabular}{ c | c }
  $(A_e + L_eC_{e2})^T$ & $C_{e1}^T$ \\ \hline
  $(B_1+L_eD_{21})^T$ & $B_1^TC_1^T$
\end{tabular}
\end{bmatrix}
\end{align*}
is PR. Therefore it follows from the Dual Positive Real Lemma in \cite{anderson1967dual} that the dual system
\begin{align*}
\begin{bmatrix}
    \begin{tabular}{ c | c }
  $A_e + L_eC_{e2}$ & $B_1+L_eD_{21}$ \\ \hline
  $C_{e1}$ & $C_1B_1$
\end{tabular}
\end{bmatrix}
\end{align*}
has a solution $V\geq0$ that satisfies the ARE
\begin{align*}
  V(A_e + L_eC_{e2}) + (A_e + L_eC_{e2})^TV + Q_v^TR^{-1}Q_v = 0,
\end{align*}
where
\begin{align*}
 Q_v &= C_{e1} -(B_1+L_eD_{21})V.
\end{align*}
\end{proof}

We now offer a proof for our main result.

\begin{proof}[Proof of Theorem~\ref{theorem: NI outputfeedback sufficiency}]
Consider the closed-loop system with realization (\ref{closed loop realization x.}). If we choose $(x^T,x^T - x_k^T)^T$ as our state vector then our new realization can be represented by the transfer function $G_{cl}(s) = \begin{bmatrix}
    \begin{tabular}{ c | r }
  $A_{cl}$ & $B_{cl}$ \\ \hline
  $C_{cl}$ & $0$
\end{tabular}
\end{bmatrix}$, where
\begin{align*}
A_{cl} &= \begin{bmatrix}
    \begin{tabular}{ c c }
  $A+B_2C_k$ & $-B_2C_k$ \\
  $A-A_k+B_2C_k-B_kC_2$ & $A_k-B_2C_k$
\end{tabular}
\end{bmatrix}, \\
B_{cl} &=
\begin{bmatrix}
    \begin{tabular}{ c }
  $B_1$ \\
  $B_1-B_kD_{21}$
\end{tabular}
\end{bmatrix}, \quad \quad \quad
C_{cl} =
\begin{bmatrix}
    \begin{tabular}{ c c}
  $C_1$ & $0$
\end{tabular}
\end{bmatrix}.
\end{align*}

We will show sufficiency by showing that $G_{cl}(s)$ with the controller given by (\ref{controller: C_k choice})-(\ref{controller: A_k choice}) satisfies the NI ARE
\begin{align} \label{math:riccati closed loop}
  X(\Sigma) &= \begin{bmatrix}
    \begin{tabular}{ l r }
  $X_{11}$ & $X_{21}^T$ \\
  $X_{21}$ & $X_{22}$
\end{tabular}
\end{bmatrix}
= \Sigma A_{cl} + A_{cl}^T\Sigma \nonumber\\
& \omit\hfill ${} + (C_{cl}A_{cl}-B_{cl}^T\Sigma)^TR^{-1}(C_{cl}A_{cl} - B_{cl}^T\Sigma) = 0$, \nonumber\\
\end{align}
with a suitable choice of $\Sigma$. Thus by Lemma~\ref{lemma: NI ARE lemma} the closed-loop system is NI. It follows from condition (c) that the controller(\ref{controller: C_k choice})-(\ref{controller: A_k choice}) is well defined. Using the definition of $P$ given in the conditions of the theorem and the definition of $V$ given in Corollary~\ref{collary: Building V}, we may define $\Sigma$ as
\begin{align}
 \Sigma &=  \begin{bmatrix}
        \begin{tabular}{ l r }
        $P$     & $0$  \\
        $0$   & $V$
      \end{tabular}
    \end{bmatrix} \geq 0.
\end{align}

We now decompose (\ref{math:riccati closed loop}) and calculate
\begin{align*}
 X_{11} &= P(A+B_2C_K) + (A+B_2C_K)^TP \\
  & \omit\hfill ${}+ \big(C_1(A+B_2C_K)-B_1^TP\big)^TR^{-1}$ \\
  & \omit\hfill ${}\cdot\big(C_1(A+B_2C_K)-B_1^TP\big),$ \\
 X_{21} &= V(A-A_k+B_2C_K-B_kC_2) - (B_2C_K)^TP\\
 & \omit\hfill ${} - \big((C_1B_2C_K)^T + V(B_1-B_kD_{21})\big)R^{-1}$ \\
  & \omit\hfill ${}\cdot\big(C_1(A+B_2C_K)-B_1^TP\big),$ \\
 X_{22} &= V(A_k-B_2C_K) + (A_k-B_2C_K)^TV \\
 & \omit\hfill ${} + \big((C_1B_2C_K)^T + V(B_1-B_kD_{21})\big)R^{-1}$ \\
 & \omit\hfill ${} \cdot\big((C_1B_2C_K)^T + V(B_1-B_kD_{21})\big)^T.$
\end{align*}

Our choice of $C_k = F$ results in $X_{11} = R(P)$ and $X_{11} = 0$ follows from condition (a). Similarly, $A_k$ and $B_k$ result in algebraic cancellation leading to $X_{21} = 0$. Finally, after substitution and appropriate algebraic manipulation we are left with $X_{22} = (\ref{math: ARE using V}) = 0$. Thus, (\ref{math:riccati closed loop}) is satisfied and $G_{cl}(s)$ is NI.
\\\\
The necessity of conditions (a),(b),(c) can be proven as follows.
Assume a strictly proper controller exists such that the closed-loop transfer function (\ref{math: closed-loop transfer function}) is strictly negative imaginary and minimal. Then by Lemma~\ref{lemma: SNI ARE lemma}, there exists a matrix
\begin{align*}
 \Sigma &=  \begin{bmatrix}
        \begin{tabular}{ l r }
        $\Sigma_{11}$     & $\Sigma_{12}$  \\
        $\Sigma_{12}^T$   & $\Sigma_{22}$
      \end{tabular}
    \end{bmatrix} > 0,
\end{align*}
such that
\begin{IEEEeqnarray}{c} \label{math: closed-loop riccati equation}
  \Sigma A_{c} + A_{c}^T \Sigma + (B_c^T\Sigma-C_cA_c)^TR^{-1}(B_c^T\Sigma-C_cA_c) = 0, \nonumber\\
\end{IEEEeqnarray}
where
\begin{align*}
 A_c &=  \begin{bmatrix}
        \begin{tabular}{ l r }
        $A$     & $B_2C_k$  \\
        $B_kC_2$   & $A_k$
      \end{tabular}
    \end{bmatrix}, &
 B_c &= \begin{bmatrix}
        \begin{tabular}{ c }
        $B_1$  \\
        $B_kD_{21}$
      \end{tabular}
    \end{bmatrix}, \\
 C_c &= \begin{bmatrix}
        \begin{tabular}{ l r }
        $C_1$  &
        $0$
      \end{tabular}
    \end{bmatrix}, &
 R &= C_cB_c + B_c^TC_c^T >0.
\end{align*}

Define the following transformation matrix $\tilde{T}$ which will be used to diagonalize our system:
\begin{align*}
 \tilde{T} &=  \begin{bmatrix}
        \begin{tabular}{ l r }
        $I$     & $0$  \\
        $\tilde{E}$   & $I$
      \end{tabular}
    \end{bmatrix}, & \tilde{E} = -\Sigma_{22}^{-1}\Sigma_{12}^T.
\end{align*}

We can use this transformation matrix to define the following matrices:
\begin{align*}
 \tilde{\Sigma} &=  \tilde{T}^T\Sigma\tilde{T} = \begin{bmatrix}
        \begin{tabular}{ l r }
        $P$     & $0$  \\
        $0$   & $\Sigma_{22}$
      \end{tabular}
    \end{bmatrix}, \\
\tilde{A} &= \tilde{T}^{-1}A_{c}\tilde{T}  = \begin{bmatrix}
        \begin{tabular}{ l r }
        $A+B_2C_K\tilde{E}$     & $*$  \\
        $*$   & $*$
      \end{tabular}
    \end{bmatrix}, \\
\tilde{B} &=  \tilde{T}^{-1}B_c = \begin{bmatrix}
        \begin{tabular}{ c }
        $B_1$ \\
        $*$
      \end{tabular}
    \end{bmatrix}, \\
\tilde{C} &=  C_cA_c\tilde{T} = \begin{bmatrix}
        \begin{tabular}{ l r }
        $C_1A + C_1B_2C_k\tilde{E}$ &
        $*$
      \end{tabular}
    \end{bmatrix},
\end{align*}
where
\begin{align}
  P &= \Sigma_{11} - \Sigma_{12}\Sigma_{22}^{-1}\Sigma_{12}^T > 0. \label{Definition of P necc proof}
\end{align}

If we pre-multiply (\ref{math: closed-loop riccati equation}) by $\tilde{T}^T$ and post-multiply by $\tilde{T}$, we are left with the following equality:
\begin{IEEEeqnarray*}{c}
  \tilde{\Sigma}\tilde{A} + \tilde{A}^T\tilde{\Sigma} + (\tilde{B}\tilde{\Sigma}-\tilde{C})^TR^{-1}(\tilde{B}\tilde{\Sigma}-\tilde{C}) = 0.
\end{IEEEeqnarray*}
The (1,1) block matrix of this equality satisfies
\begin{IEEEeqnarray}{c} \label{math: condition (1) state-feedback riccati equation}
  P(A+B_2C_K\tilde{E}) + (A+B_2C_K\tilde{E})^TP + \bar{Q}^TR^{-1}\bar{Q} = 0, \nonumber\\
\end{IEEEeqnarray}
 where
\begin{align*}
  \bar{Q} &= B_1P-C_1A-C_1B_2C_k\tilde{E}.
\end{align*}
By choosing $F = C_k\tilde{E}$ we can conclude that $P> 0$ and $F$ satisfy condition (a).
\\\\
Now in order to show condition (b) is satisfied, define the following transformation matrix
\begin{align*}
 \bar{T} &=  \begin{bmatrix}
        \begin{tabular}{ l r }
        $I$     & $-\bar{E}$  \\
        $0$   & $I$
      \end{tabular}
    \end{bmatrix}, & \bar{E} = \Sigma_{11}^{-1}\Sigma_{12}.
\end{align*}

We use this transformation matrix to define the following matrices:
\begin{align*}
 \bar{\Sigma} &=  \bar{T}^T\Sigma\bar{T} = \begin{bmatrix}
        \begin{tabular}{ l r }
        $\bar{\Sigma}_{11}$     & $0$  \\
        $0$   & $*$
      \end{tabular}
    \end{bmatrix}, \\
\bar{A} &= \bar{T}^{-1}A_{c}\bar{T}  = \begin{bmatrix}
        \begin{tabular}{ l r }
        $A+\bar{E}B_KC_2$     & $*$  \\
        $*$   & $*$
      \end{tabular}
    \end{bmatrix}, \\
\bar{B} &=  \bar{T}^{-1}B_c = \begin{bmatrix}
        \begin{tabular}{ c }
        $B_1+\bar{E}B_kD_21$ \\
        $0$
      \end{tabular}
    \end{bmatrix}, \\
\bar{C} &=  C_cA_c\bar{T} = \begin{bmatrix}
        \begin{tabular}{ l r }
        $C_1A$ &
        $*$
      \end{tabular}
    \end{bmatrix}.
\end{align*}

If we pre-multiply (\ref{math: closed-loop riccati equation}) by $\bar{T}^T$ and post-multiply by $\bar{T}$, we are left with the following equality:
\begin{align*}
  \bar{\Sigma}_{11}\bar{A} + \bar{A}^T\bar{\Sigma}_{11} + (\bar{B}^T\bar{\Sigma}_{11}-\bar{C})^TR^{-1}(\bar{B}^T\bar{\Sigma}_{11}-\bar{C}) = 0.
\end{align*}
We now define the matrix
\begin{align*}
 \bar{Z} = \bar{\Sigma}^{-1} &=  \begin{bmatrix}
        \begin{tabular}{ l r }
        $\bar{\Sigma}_{11}^{-1}$     & $0$  \\
        $0$   & $*$
      \end{tabular}
    \end{bmatrix}
\end{align*}
and by Lemma~\ref{lemma: Z equals P inverse} the following equality is also satisfied:
\begin{align*}
  \bar{Z}\bar{A}^T+\bar{A}\bar{Z}+(\bar{B}^T-\bar{C}\bar{Z})^TR^{-1}(\bar{B}^T-\bar{C}\bar{Z}) = 0.
\end{align*}
Thus, the (1,1) block matrix of this equality satisfies
\begin{align*}
  Z(A+\bar{E}B_KC_2)^T + (A+\bar{E}B_KC_2)Z + \bar{Q}R^{-1}\bar{Q}^T= 0,
\end{align*}
 where
\begin{align}
  Z &= \bar{\Sigma}_{11}^{-1} > 0, \label{Definition of Z necc proof} \\
  \bar{Q} &= B_1+\bar{E}B_kD_21 - ZA^TC_1^T. \nonumber
\end{align}
By choosing $L = \bar{E}B_k$ we can conclude that $Z>0$ and $L$ satisfy condition (b).
\\\\
Finally, note that with $P$ and $Z$ defined as in (\ref{Definition of P necc proof}) and (\ref{Definition of Z necc proof}), we have
\begin{align*}
	Z^{-1} - P = & \quad {\Sigma}_{11} - (\Sigma_{11} - \Sigma_{12}\Sigma_{22}^{-1}\Sigma_{12}^T), \\
  =& \quad \Sigma_{12}\Sigma_{22}^{-1}\Sigma_{12}^T \geq 0, \\
  =& \quad Z^{-1}(Z - ZPZ)Z^{-1} \geq 0.
\end{align*}
Thus, condition (c) is satisfied.
\end{proof}


\bibliographystyle{IEEEtran}
\bibliography{root}


\end{document}